\documentclass[11pt,amstex,amssymb]{amsart}
\usepackage{amsmath,amsthm,amsfonts,amssymb,amscd}
\usepackage[latin1]{inputenc}
\usepackage[all]{xy}
\usepackage[dvips]{graphicx}
\usepackage{amsmath}
\usepackage{amsthm}
\usepackage{amsfonts}
\usepackage{amssymb}%
\usepackage{amsmath}
\usepackage{amssymb}
\usepackage{amsthm}

\def\fa{{\mathcal{F}}}
\def\O{{\mathcal{O}}}

\def\M{{\mathcal{M}}}
\def\D{{\mathcal{D}}}
\def\U{{\mathcal{U}}}
\def\B{{\mathcal{B}}}

\def\ve{{\varepsilon}}
\def\bq{{\mathbb{Q}}}

\def\rg{{\rangle}}

\newtheorem{theorem}{Theorem}

\theoremstyle{plain}
\newtheorem{Theorem}{Theorem}[section]
\newtheorem{Lemma}[Theorem]{Lemma}
\newtheorem{Corollary}[Theorem]{Corollary}
\newtheorem{Proposition}[Theorem]{Proposition}

\theoremstyle{definition}
\newtheorem{Remark}[Theorem]{Remark}
\newtheorem{Definition}[Theorem]{Definition}

\newtheorem{Claim}[Theorem]{Claim}

\newtheorem{Question}[Theorem]{Question}

\title[On the integrability of  Hill's equation]{
On the integrability  of Hill's equation of the motion of the moon}

\author{F.  Reis}
\author{B.  Scárdua}


\address{}

\begin{document}

\begin{abstract}
We study under the standpoint of integrable complex analytic 1-forms
(complex analytic foliations), a class of second order ordinary
differential equations with periodic coefficients. More precisely,
we study Hill's equations of motion of the moon, which are related
to the dynamics of the system Sun-Earth-Moon. We associate to the
{\em complex Hill equation} an integrable complex analytic one-form
in dimension three. This defines a {\it Hill foliation}. The
existence of first integral for a Hill foliation is then studied.
The simple cases correspond to the existence of  rational or
Liouvillian first integrals. We then  prove the existence of a {\it
Bessel type} first integral in a more general case. We construct a
standard two dimensional model for the foliation which we call {\it
Hill fundamental form}. This plane foliation is then studied also
under the standpoint of reduction of singularities and existence of
first integral. For the more general case of the Hill equation, we
prove for the corresponding Hill foliation, the existence of a
Laurent-Fourier type formal first integral. Our approach suggests
that there may be a class of plane foliations admitting Bessel type
first integrals, in connection with the classification of (holonomy)
groups of germs of complex diffeomorphisms associate to a certain
class of second order ODEs.
 \end{abstract}

\maketitle

\newcommand\virt{\rm{virt}}
\newcommand\SO{\rm{SO}}
\newcommand\G{\varGamma}
\newcommand\Om{\Omega}
\newcommand\Kbar{{K\kern-1.7ex\raise1.15ex\hbox to 1.4ex{\hrulefill}}}
\newcommand\codim{\rm{codim}}
\renewcommand\:{\colon}
\newcommand\s{\sigma}
\def\vol#1{{|{\bfS}^{#1}|}}

\def\fa{{\mathcal F}}
\def\H{{\mathcal H}}
\def\O{{\mathcal O}}
\def\P{{\mathcal P}}
\def\L{{\mathcal L}}
\def\C{{\mathcal C}}
\def\Z{{\mathcal Z}}

\def\M{{\mathcal M}}
\def\N{{\mathcal N}}
\def\R{{\mathcal R}}
\def\ea{{\mathcal e}}
\def\Oa{{\mathcal O}}
\def\ee{{\bfE}}

\def\A{{\mathcal A}}
\def\B{{\mathcal B}}
\def\H{{\mathcal H}}
\def\V{{\mathcal V}}
\def\U{{\mathcal U}}
\def\al{{\alpha}}
\def\be{{\beta}}
\def\ga{{\gamma}}
\def\Ga{{\Gamma}}
\def\om{{\omega}}
\def\Om{{\Omega}}
\def\La{{\Lambda}}
\def\ov{\overline}
\def\dd{{\bfD}}
\def\pp{{\bfP}}

\def\nn{{\mathbb N}}
\def\zz{{\mathbb Z}}
\def\bq{{\mathbb Q}}
\def\bp{{\mathbb P}}
\def\bd{{\mathbb D}}
\def\bh{{\mathbb H}}
\def\te{{\theta}}
\def\rr{{\mathbb R}}
\def\bb{{\mathbb B}}

\def\pp{{\mathbb P}}

\def\dd{{\mathbb D}}
\def\zz{{\mathbb Z}}
\def\qq{{\mathbb Q}}

\def\hh{{\mathbb H}}
\def\nn{{\mathbb N}}

\def\LL{{\mathbb L}}

\def\co{{\mathbb C}}
\def\qq{{\mathbb Q}}
\def\na{{\mathbb N}}
\def\esima{${}^{\text{\b a}}$}
\def\esimo{${}^{\text{\b o}}$}
\def\rg{\rangle}
\def\ro{{\rho}}
\def\lV{\left\Vert}
\def\rV{\right\Vert }
\def\lv{\left|}
\def\rv{\right| }
\def\Sa{{\mathcal S}}
\def\D{{\mathcal D  }}

\def\si{{\bf S}}
\def\ve{\varepsilon}
\def\vr{\varphi}
\def\lV{\left\Vert }
\def\rV{\right\Vert}
\def\lv{\left| }
\def\rv{\right|}
\def\Range{\rm{{R}}}
\def\vol{\rm{{Vol}}}
\def\ind{\rm{{i}}}

\def\Int{\rm{{Int}}}
\def\Dom{\rm{{Dom}}}
\def\supp{\rm{{supp}}}
\def\Aff{\rm{{Aff}}}
\def\Exp{\rm{{Exp}}}
\def\Hom{\rm{{Hom}}}
\def\codim{\rm{{codim}}}
\def\cotg{\rm{{cotg}}}
\def\dom{\rm{{dom}}}
\def\Sa{\mathcal{{S}}}

\def\VIP{\rm{{VIP}}}
\def\argmin{\rm{{argmin}}}
\def\Sol{\rm{{Sol}}}
\def\Ker{\rm{{Ker}}}
\def\Sat{\rm{{Sat}}}
\def\diag{\rm{{diag}}}
\def\rank{\rm{{rank}}}
\def\Sing{\rm{{Sing}}}
\def\sing{\rm{{sing}}}
\def\hot{\rm{{h.o.t.}}}

\def\Fol{\rm{{Fol}}}
\def\grad{\rm{{grad}}}
\def\id{\rm{{id}}}
\def\Id{\rm{{Id}}}
\def\sep{\rm{{Sep}}}
\def\Aut{\rm{{Aut}}}
\def\Sep{\rm{{Sep}}}
\def\Res{\rm{{Res}}}
\def\ord{\rm{{ord}}}
\def\h.o.t.{\rm{{h.o.t.}}}
\def\Hol{\rm{{Hol}}}
\def\Diff{\rm{{Diff}}}
\def\SL{\rm{{SL}}}

\tableofcontents

\section{Introduction}
 In the year of  1877, G.W. Hill published his celebrated work
(\cite{hill}) \textit{``On the part of the motion of the lunar
perigee which is a function of the mean motions of the sun and
moon"}.  In this masterpiece   Hill describes the movement of the
moon around the earth by considering it as a harmonic oscillator in
a periodic gravitational vector field and introduced the following
model (called {\it Hill's equation} \cite{arnold0,arnold})
\begin{eqnarray} \label{HEqIntro}
u'' + p(z)u = 0
\end{eqnarray}
where $p$ is a periodic function of the time $t$ and $u(t)$
describes the distance (position) of the moon with respect to the
earth.  Hill introduced and applied successfully for the first time
the theory of {\it infinite determinants}. From that moment on, a
great effort has been made towards the comprehension of  Hill's
equation (see for instance \cite{magnusWinkler}).

In this work we study the {\it complex Hill equation}. In short,
this means we shall study the equation $u''(z) + p(z)u(z) = 0$,
where $p(z)$ is a complex analytic periodic function defined in some
domain $A\subseteq \mathbb C$. Since this equation is related to the
problem of planetary movement and stability of the solar system, we
address the following question:
\begin{Question}
Is there any integrable structure  connected to  the Hill equation?
\end{Question}
The above question may seem quite general. Indeed, one of the gains
of this work is to investigate an appropriate notion of
integrability in this case. Usually, when dealing with ODEs the word
{\em integrability} means to find some  kind of potential function,
but this does not necessarily seem appropriate because the order of
the equation is two. In this work we bring some geometrical features
based in the theory of complex analytic (holomorphic) foliations. In
this foliation framework the notion of integrability has essentially
two interpretations. One is the existence of an integrable one-form
to which the ODE is tangent (cf.
Theorem~\ref{Theorem:integrableHill}). Another is the existence of a
first integral for a foliation that is tangent to the ODE (cf.
Theorems~\ref{Theorem:integralhillform} and
\ref{Theorem:LaurentFourier}). We shall explore these concepts and
try to shed some light into the comprehension of the Hill equation
in the complex framework.

\section{The complex Hill equation} \label{Sec EqHill}
In his work
 \cite{hill}, G. W. Hill considered
 equations related to the  {\it three-body problem}, namely
 earth, moon and sun, given by
\begin{eqnarray}
\frac{d^2u}{d z^2} - 2m\frac{dv}{d z} + \chi \frac{u}{r^3} &=& 3m^2u \label{EqHillOrigin1}\\
\frac{d^2v}{d z^2} - 2m\frac{du}{d z} + \chi \frac{v}{r^3} &=& 0
\label{EqHillOrigin2}
\end{eqnarray}
In the above equations, $u,v$ are given  rectangular coordinates of
the moon having the earth as center and, $r = \sqrt{u^2 + v^2}$. The
parameter  $m$ is given by
$$m = \frac{n'}{n-n'},$$ where  $n'$ is the mean motion of the sun and  $n$ is
the mean motion  of the moon.  Hill works with the estimative
 $m = 0,08084893679$. The parameter $\chi$ is given by
\begin{eqnarray*}
\chi = G\frac{M_e + M_m}{(n-n')^2},
\end{eqnarray*}
where  $M_e,M_m$ are the mass of the earth and of the moon
respectively, and $G$ is  Cavendish's gravitational constant.

G. W. Hill was able to associate to equations (\ref{EqHillOrigin1})
and  (\ref{EqHillOrigin2}) the single equation,
\begin{equation} \label{EqHillreal}
u''(s) + p(s)u(s) = 0
\end{equation}
where  $p$ is a periodic real valued function. The variable  $s$ is
related to the time $t$ by the formula  $s = (n - n')(t - t_0)$,
where  $t_0$ is the initial time. Most of the classical works in
Hill's equation are based on the following hypothesis:

\vglue.1in \textbf{Hypothesis 1: } {\em $p(s)$ is a real integrable
periodic function of period $\pi$.} \vglue.1in

The original work of Hill and the classical reference
\cite{magnusWinkler} of  Magnus and  Winkler consider Hypothesis  1
(see for instance \cite{harry}). Despite this restriction, the
solutions are allowed to have complex values.

In this work we shall start the study of the {\it complex Hill
equation}
\begin{eqnarray} \label{HEqComplex}
u''(z) + p(z)u(z) = 0
\end{eqnarray}
where $p(z)$ is a complex periodic function. We shall assume that
$p(z)$ is complex analytic defined in a strip $A\subseteq \mathbb C$
containing the real axis $\Im(z)=0$.

Depending on the viewpoint, the complex case and
 the real case may be  quite different, for some functions are
periodic in the complex framework, but not in the real setting. This
is the case of the exponential function $p(z)=e^z, z \in \mathbb C$.
Another important particular case is the {\it complex  Mathieu
equation} (\cite{arnold,tongues})

\begin{eqnarray} \label{HEqComplexMathieu} u''(z) + (a+  b\cos z) u(z) = 0
\end{eqnarray}
where $a,b \in \mathbb C$ are complex numbers. For this function the
main difference with the real case is the fact that the real
function is bounded but this is not the case of the complex
function. The complex Mathieu equation  will be treated in a
forthcoming work.

\section{ Hill forms and Hill foliations}
In this section we shall study the Hill equation (\ref{HEqComplex})
from the point of view of integrable one-forms.  We start with a
more general situation of a second order complex ODE
\begin{eqnarray} \label{eq2ordem}
u''(z) + b(z)u = 0
\end{eqnarray}
where  $b$ is a complex analytic function defined in an open set
$U\subset \mathbb C$.  As a first step we perform the classical
order reduction process where equation  (\ref{eq2ordem}) is
rewritten after the following ``change of coordinates": $x = u, \, y
= u', \, z= z$. We then obtain $x'= u' = y, \, y' = u'' = - b(z)u =
- b(z)x, \, z' = 1$. Therefore, a natural vector field $X\colon
\mathbb{C}^2 \times U \rightarrow \mathbb{C}^2\times U$ associated
to  equation (\ref{eq2ordem}) is   given by $X(x,y,z) = y
\frac{\partial}{\partial x} - b(z)x \frac{\partial}{\partial y} +
\frac{\partial}{\partial z}.$ We shall refer to this vector field as
the {\it Hill vector field} associated to the Hill equation
(\ref{eq2ordem}).

We shall say that a complex vector field $X$ in some open subset
$W\subset \mathbb C^3$ is {\it integrable in the weak sense} or {\it
w-integrable} for short,  if it is tangent to a codimension one
complex analytic foliation (possibly singular) $\fa$ of $W$. The
vector field is said to be {\it s-integrable} or {\it integrable in
the strong sense} if it is w-integrable and the foliation $\fa$ can
be chosen to have a first integral. For the moment we shall not
decide what kind of first integral we shall be working with
(polynomial, rational, meromorphic, Liouvillian, formal...).

\begin{theorem} \label{Theorem:integrableHill}
A Hill vector field is always w-integrable.
\end{theorem}

For the proof of Theorem~\ref{Theorem:integrableHill} we shall need
the following technical result:

\begin{Lemma} \label{eqInteg}
Let $X:\mathbb{C}^3 \rightarrow \mathbb{C}^3$ be a complex vector
field in  $\mathbb{C}^3$ given by $ X(x,y,z) =
f_1(x,y,z)\frac{\partial}{\partial x} +
f_2(x,y,z)\frac{\partial}{\partial y} +
f_3(x,y,z)\frac{\partial}{\partial z}$ where  $f_j : \mathbb{C}^3
\rightarrow \mathbb{C}$ is complex analytic $j=1,2,3. $ Let also
$\omega$ be a complex analytic 1-form given in $\mathbb C^3$  by $
\omega(x,y,z) = A(x,y,z)dx + B(x,y,z)dy + C(x,y,z)dz. $ and
satisfying $\omega(X)\equiv 0.$
 Then we have $\omega \wedge d\omega = 0$
provided that $ f_1f_3B\frac{\partial A}{\partial x} +
f_2f_3B\frac{\partial A}{\partial y} + f_3^2B\frac{\partial
A}{\partial z} = \left( f_3 \frac{\partial f_1}{\partial y} - f_1
\frac{\partial f_3}{\partial y} \right)A^2
+ \left[f_1f_3 \frac{\partial B}{\partial x} + f_2f_3\frac{\partial B}{\partial y} - f_3^2\frac{\partial B}{\partial z} + B\left[f_1\frac{\partial f_3}{\partial x} - f_2\frac{\partial f_3}{\partial y} - f_3\frac{\partial f_1}{\partial x} + f_3\frac{\partial f_2}{\partial y}\right]\right]A\\
+ B^2 \left[ f_2f_3 \frac{\partial f_3}{\partial x} -
f_3\frac{\partial f_2}{\partial x}\right]. $
\end{Lemma}

\begin{proof}
From  $\omega(X) = 0$ we have $ A f_1 + B f_2 + C f_3 = 0.$ Hence,
\begin{eqnarray} \label{expC}
C = -\frac{(A f_1 + B f_2)}{f_3}.
\end{eqnarray}
The partial derivatives of $C$ can be calculated with the aid of
(\ref{expC}) resulting in
\begin{eqnarray}\label{DCx}
\frac{\partial C}{\partial x} = -\frac{f_1}{f_3}\frac{\partial
A}{\partial x} -\frac{1}{f_3}\frac{\partial f_1}{\partial x}A
-\frac{f_2}{f_3} \frac{\partial B}{\partial x} -
\frac{1}{f_3}\frac{\partial f_2}{\partial x}B +
\frac{f_1}{f_3^2}\frac{\partial f_3}{\partial x}A +
\frac{f_2}{f_3^2}\frac{\partial f_3}{\partial x}B
\end{eqnarray}
and,
\begin{eqnarray}\label{DCy}
\frac{\partial C}{\partial y} = -\frac{f_1}{f_3}\frac{\partial
A}{\partial y} -\frac{1}{f_3}\frac{\partial f_1}{\partial y}A
-\frac{f_2}{f_3} \frac{\partial B}{\partial y} -
\frac{1}{f_3}\frac{\partial f_2}{\partial y}B +
\frac{f_1}{f_3^2}\frac{\partial f_3}{\partial y}A +
\frac{f_2}{f_3^2}\frac{\partial f_3}{\partial y}B.
\end{eqnarray}
On the other hand, from the expression of $\omega$ we have
\begin{eqnarray*}
d\omega &=& dA\wedge dx + dB\wedge dy + dC\wedge dz \\
&=& \left( \frac{\partial B}{\partial x} -  \frac{\partial
A}{\partial y}  \right)dx\wedge dy + \left( \frac{\partial
C}{\partial x} -  \frac{\partial A}{\partial z}  \right)dx\wedge dz
+ \left( \frac{\partial C}{\partial y} -  \frac{\partial B}{\partial
z}  \right)dy\wedge dz.
\end{eqnarray*}
Thus,
\begin{eqnarray*}
\omega \wedge d\omega &=& A\left( \frac{\partial C}{\partial y} -  \frac{\partial B}{\partial z}  \right)dx\wedge dy\wedge dz + B\left( \frac{\partial C}{\partial x} -  \frac{\partial A}{\partial z}  \right)dy\wedge dx\wedge dz \\
&+& C\left( \frac{\partial B}{\partial x} -  \frac{\partial A}{\partial y}  \right)dz\wedge dx\wedge dz \\
&=& \left[ A\left( \frac{\partial C}{\partial y} -  \frac{\partial
B}{\partial z}  \right) - B\left( \frac{\partial C}{\partial x} -
\frac{\partial A}{\partial z}  \right)  + C\left( \frac{\partial
B}{\partial x} -  \frac{\partial A}{\partial y}  \right) \right]
dx\wedge dz \wedge dz.
\end{eqnarray*}
Then, $\omega\wedge d\omega = 0$ implies $ A\frac{\partial
C}{\partial y} - A\frac{\partial B}{\partial z} - B \frac{\partial
C}{\partial x} + B\frac{\partial A}{\partial z}  + C\frac{\partial
B}{\partial x} - C\frac{\partial A}{\partial y}  = 0.$ Replacing
(\ref{DCx}) and (\ref{DCy}) in the above expression we obtain
\begin{eqnarray*}
&-&\frac{f_1}{f_3}\frac{\partial A}{\partial y}A - \frac{1}{f_3}\frac{\partial f_1}{\partial y}A^2 - \frac{f_2}{f_3}\frac{\partial B}{\partial y}A - \frac{1}{f_3}\frac{\partial f_2}{\partial y} BA + \frac{f_1}{f_3^2} \frac{\partial f_3}{\partial y} A^2 + \frac{f_2}{f_3^2}\frac{\partial f_3}{\partial y} BA \\
&+& \frac{f_1}{f_3}B \frac{\partial A}{\partial x} + \frac{1}{f_3}\frac{\partial f_1}{\partial x}BA +\frac{f_2}{f_3}\frac{\partial B}{\partial x}B + \frac{1}{f_3}\frac{\partial f_2}{\partial x}B^2 - \frac{f_1}{f_3^2}\frac{\partial f_3}{\partial x}BA - \frac{f_2}{f_3^2}\frac{\partial f_3}{\partial x}B^2 \\
&+& B\frac{\partial A}{\partial z} - \frac{f_1}{f_3}\frac{\partial
B}{\partial x}A - \frac{f_2}{f_3}\frac{\partial B}{\partial x}B +
\frac{f_1}{f_3}\frac{\partial A}{\partial y}A + \frac{f_2}{f_3}B
\frac{\partial A}{\partial y} + A\frac{\partial B}{\partial z} = 0.
\end{eqnarray*}
Simplification of this equation then gives
\begin{eqnarray*}
&-& \frac{1}{f_3}\frac{\partial f_1}{\partial y}A^2 - \frac{f_2}{f_3}\frac{\partial B}{\partial y}A - \frac{1}{f_3}\frac{\partial f_2}{\partial y} BA + \frac{f_1}{f_3^2} \frac{\partial f_3}{\partial y} A^2 + \frac{f_2}{f_3^2}\frac{\partial f_3}{\partial y} BA \\
&+& \frac{f_1}{f_3}B \frac{\partial A}{\partial x} + \frac{1}{f_3}\frac{\partial f_1}{\partial x}BA + \frac{1}{f_3}\frac{\partial f_2}{\partial x}B^2 - \frac{f_1}{f_3^2}\frac{\partial f_3}{\partial x}BA - \frac{f_2}{f_3^2}\frac{\partial f_3}{\partial x}B^2 \\
&+& B\frac{\partial A}{\partial z} - \frac{f_1}{f_3}\frac{\partial
B}{\partial x}A + \frac{f_2}{f_3}B \frac{\partial A}{\partial y} +
A\frac{\partial B}{\partial z} = 0.
\end{eqnarray*}
Reorganizing the terms we get
\begin{eqnarray*}
\nonumber\frac{f_1}{f_3}B\frac{\partial A}{\partial x} &+& \frac{f_2}{f_3}B\frac{\partial A}{\partial y} + B\frac{\partial A}{\partial z}\\ &=& \left( \frac{1}{f_3} \frac{\partial f_1}{\partial y} - \frac{f_1}{f_3^2} \frac{\partial f_3}{\partial y} \right)A^2  \\
&+&\left[\frac{f_1}{f_3} \frac{\partial B}{\partial x} + \frac{f_2}{f_3}\frac{\partial B}{\partial y} - \frac{\partial B}{\partial z} + B\left[\frac{f_1}{f_3^2} \frac{\partial f_3}{\partial x} - \frac{f_2}{f_3^2} \frac{\partial f_3}{\partial y} - \frac{1}{f_3}\frac{\partial f_1}{\partial x} + \frac{1}{f_3}\frac{\partial f_2}{\partial y}\right]\right]A\\
&+& B^2 \left[ \frac{f_2}{f_3} \frac{\partial f_3}{\partial x} -
\frac{1}{f_3} \frac{\partial f_2}{\partial x}\right].
\end{eqnarray*}
Multiplication by $f_3^2$ ends the proof.
\end{proof}

\begin{proof}[Proof of Theorem~\ref{Theorem:integrableHill}]
We look for an integrable complex analytic one-form $\omega$ in
$\mathbb C^2 \times U$ such that $\omega(X)=0$. Let us write
$\omega(x,y,z) = A(x,y,z)dx + B(x,y,z)dy + C(x,y,z)dz $ where the
coefficients $A,B,C$ are complex analytic in $\mathbb C^2\times U$.
We know from Lemma~\ref{eqInteg} that $\omega = A(x,y,z)dx +
B(x,y,z)dy + C(x,y,z)dz$ satisfying  $\omega(X) \equiv 0$
 also satisfies  $\omega\wedge d\omega=0$
provided that
\begin{eqnarray*}
yB\frac{\partial A}{\partial x} -   [b(z)x]B\frac{\partial
A}{\partial y} + B\frac{\partial A}{\partial z}  = A^2  + \left[y
\frac{\partial B}{\partial x} - [b(z)x]\frac{\partial B}{\partial y}
- \frac{\partial B}{\partial z}\right]A + b(z)B^2
\end{eqnarray*}
and
\begin{eqnarray} \label{C}
C = -(-yA - b(z)xB)= yA + xp(z)B.
\end{eqnarray}
Let us take  $B\equiv 1$ in the last equation obtaining
\begin{eqnarray} \label{eqInt2lin}
y\frac{\partial A}{\partial x} - [b(z)x]
\frac{\partial A}{\partial y} + \frac{\partial A}{\partial z} = A^2  + b(z).
\end{eqnarray}

Equation~\eqref{eqInt2lin} is a semilinear first order PDE, which
can be solved in real case by the classical {\em method of
characteristics}. Let us try this same method in our complex
framework. For this we introduce the following system:

\[
\frac{dx}{dt} = y, \, \frac{dy}{dt}= - b(z)x, \, \frac{dz}{dt}= 1,
\, \frac{d\phi}{dt}=A^2 + b
\]

Then, from this we obtain:

\begin{Claim} \label{verifiq}
Equation  \emph{(\ref{eqInt2lin})} admits the solution $A =
\frac{-y}{x}$.
\end{Claim}

In fact, for $A = \frac{-y}{x}$ we have $\frac{\partial A}{\partial
x} = \frac{y}{x^2}$, $\frac{\partial A}{\partial y} = -\frac{1}{x}$
e $\frac{\partial A}{\partial z} = 0$. Substituting in
(\ref{eqInt2lin}) we have
\[
y\frac{y}{x^2} - [a(z)y + b(z)x](-\frac{1}{x}) + 0 = \frac{y^2}{x^2} - a(z)\frac{-y}{x}+ b(z).
\]
This proves Claim~\ref{verifiq}.

Hence, substituting the values of $A$ and $B$ in \eqref{C}, we have
$C = \frac{y^2}{x} + b(z)x$. Thus, the meromorphic 1-form $\Omega:=
-\frac{y}{x}dx + dy + \left[ \frac{y^2}{x} + b(z)x \right] dz$
satisfies $\Omega(X)=0$. Multiplying $\Omega$ by its poles we obtain
$\omega =x. \Omega= -ydx + xdy + [y^2  + b(z)x^2]dz,$ which is a
complex analytic solution to the integrability problem, proving the
theorem.
\end{proof}

From now on we shall consider the case where the function $b$ is an
entire  periodic function $b=p(z)$ ie., there is a complex number $T
\in \mathbb C\setminus \{0\}$ such that  $p(z+T) = p(z)$. Therefore,
according to Theorem~\ref{Theorem:integrableHill} we can define a
complex analytic foliation $\mathcal H$ associate to the  Hill
equation (\ref{HEqComplex}).
\begin{Definition} \rm{
We shall refer to the foliation  $\mathcal H$ in $\mathbb C^3$
defined by
\begin{eqnarray*}
\omega_{\H} = -ydx + xdy + [y^2 + p(z)x^2]dz = 0
\end{eqnarray*}
as  \emph{Hill foliation} of parameter $p$, where  $p$ is a periodic
complex analytic function of period  $T$ (real or complex).}
\end{Definition}

A first question involving Hill foliations is the following:
\begin{Question}
Is there any type of first integral for a Hill foliation?

\end{Question}

Let us begin with the very basic cases.

\vglue.1in \noindent{Case $p(z)=0$}. In this case we have the
corresponding Hill foliation  $-ydx + xdy + y^2dz = 0$. This
foliation exhibits a rational first integral given by $y/x + z$.

\vglue.1in \noindent{\bf Case  $p(z)=1$}. Here  we have $\omega=- y
dx + x dy + (y^2 + x^2 ) dz= x^2 \big[ d(y/x) + ((y/x)^2 +
1)dz\big]$. This form admits a Liouvillian first integral. Indeed,
$\omega/[x^2(1+ (y/x)^2]$ is closed and rational and can be
integrated by logarithmic.

The interesting cases appear when we consider $p(z)$ a non-constant
periodic function. In this situation we shall ask for the existence
of a first integral for the corresponding Hill foliation.

\section{The Hill fundamental form}
In this section we study the special case of the Hill equation where
the periodic function is the exponential function, $p(z)= e^z$. From
this special case we shall obtain a fundamental form that shall play
an important role in our study of the general case. We start with
the differential form given above $\omega = -ydx + xdy + [y^2 + e^z
x^2]dz.$  Notice that
\[
\frac{1}{x^2} \omega = \frac{-ydx + xdy}{x^2}  + [\frac{y^2}{x^2}  +
e^z]dz = d(\frac{y}{x}) + [(\frac{y}{x})^2 + e^z)dz.
\]

Let us call $\phi=-y/x, \psi = e^z$. Then $dz=d\psi/\psi$ and we can
rewrite
\[
\frac{1}{x^2} \omega = d(\frac{y}{x}) + [(\frac{y}{x})^2 + e^z)dz =
-d\phi + [\phi^2 + \psi]\frac{d\psi}{\psi} =
\frac{1}{\psi}\big[-\psi d\phi + [\phi^2 + \psi]d\psi\big].
\]

In a modern language, if we define the rational map $\Pi\colon
\mathbb C^3 \dashrightarrow \mathbb P^1 \times \mathbb P^1$ by
$\Pi(x,y,z) = (\phi, \psi)=(-\frac{y}{x},e^z)$ then we shall state:
\begin{Lemma}
The Hill foliation $\mathcal H$ on $\mathbb C^3$ given by
\begin{eqnarray*}
 -ydx + xdy + [y^2 + e^z x^2]dz=0
\end{eqnarray*}
is the  pull-back by the rational map $\Pi$ above of the
two-dimension foliation $\mathcal H_ 2$ on $\mathbb P^1 \times
\mathbb P^1$ given on affine coordinates $(x,y)$ by
\begin{eqnarray*}
 -y dx + (x^2 + y)dy=0
\end{eqnarray*}
\end{Lemma}

\begin{Definition}
{\rm The 1-form
\begin{eqnarray}\label{OmegaExp}
\Omega_{2} = - ydx + \left(x^2 + y\right)dy
\end{eqnarray}
will be called the {\it Hill fundamental form} or just {\it Hill
form} on $\mathbb C^2$. } \end{Definition}

It is our belief that the Hill form given in (\ref{OmegaExp}) will
play a key role in the study of  Hill foliations. In what follows we
give a more detailed study of this form. We shall prove that it does
not admit  first integral of Liouvillian type, but it admits a first
integral which is given in terms of Bessel functions. Bessel
functions are described in details in \cite{watson}. In what follows
we present a short summary of their properties. Details and proofs
can be found in \cite{watson}.

\subsection{Bessel functions}
The Bessel functions appear as solutions to the classical {\it
Bessel equation}
\begin{eqnarray*}
z^2u'' + zu' + (z^2 -k^2)u = 0.
\end{eqnarray*}
These equations are related to problems in physics as vibrations,
diffusion and electromagnetic   waves. Originally defined by
Bernoulli they were formally introduced by F. Bessel in 1824 while
doing his astronomy research.
\begin{Definition}\rm{ The
function
\begin{eqnarray*}
J_{\nu}(z) = \sum_{k=0}^{\infty} \frac{(-1)^{k}}{\Gamma(k+1)\Gamma(\nu + k + 1)} \left(\frac{z}{2}\right)^{2k+\nu}
\end{eqnarray*}
is called \emph{Bessel function of first type, and order $\nu$}.
Here  $\Gamma$ is the {\it Gamma function} defined by
\begin{eqnarray*}
\Gamma(s)= \int_{0}^{\infty} e^{-t} t^{s-1} dt
\end{eqnarray*}
where  $\nu$  is a number (real or complex).}
\end{Definition}

\begin{Proposition}\label{PropBessel}
The first type Bessel functions satisfy the following properties for
each  $\nu \in \mathbb{C}$:
\begin{eqnarray*}
2J^\prime_{\nu}(z)= J_{\nu-1}(z) - J_{\nu +1}(z), \, \, J_{\nu
+1}(z) = \frac{2\nu}{z}J_{\nu}(z) - J_{\nu -1}(z);
\end{eqnarray*}
Moreover, for every $n \in \mathbb{Z}$ we have: $J_{-n}(z) =
(-1)^nJ_{n}(z),$ where $J^\prime$ denotes the usual derivative with
respect to  $z$.
\end{Proposition}
\begin{Remark}
The functions $J_{-\nu},J_{\nu}$ are known to be linearly
independent in the case where $\nu \not\in \mathbb{Z}$.
\end{Remark}

\begin{Lemma}
The Bessel functions $J_0$ and $J_1$ of first type and order $0$ and
$1$, are entire functions of the complex variable $z$.
\end{Lemma}
\begin{proof}
Indeed, $J_0(z)= \sum\limits_{k=0}^\infty \frac{(-1) ^k
z^{2k}}{2^{(2k)} (n!)^2}$. The ratio test shows that the convergence
radius of this power series is $\infty$. By its turn we have
$J_1(z)= \sum\limits_{k=0}^\infty \frac{(-1)^{k}
z^{(2k+1)}}{2^{(2k+1)}k! (k+1)!}$. Again the ratio test shows the
convergence in the whole complex plane.
\end{proof}

\begin{Definition}\rm{
The {\it second type Bessel function of order} $\nu\in \mathbb C$ is
defined by the expression
\begin{eqnarray*}
Y_{\nu}(z) = \frac{J_{\nu}(z) \cos(\nu\pi) - J_{-\nu}(z)}{\sin(\nu\pi)}
\end{eqnarray*}
where $\nu \not\in \mathbb{Z}$. Here $J_{\nu}, J_{-\nu}$ are first
type Bessel functions as defined above. In case  $\nu=n \in
\mathbb{Z}$ the corresponding second type Bessel function is defined
as
\begin{eqnarray*}
Y_{n}(z) = \lim_{\nu \rightarrow n}Y_{\nu}(z).
\end{eqnarray*}
}
\end{Definition}

The above limit exists. For instance,  $Y_{0}(0)$ is well-defined by
the well-known L'Hospital rule   (\cite{koronev} pg. 9).

\begin{Remark} \label{remarkBesselType}\rm{
Proposition~\ref{PropBessel} also holds for second type Bessel
functions.}
\end{Remark}

\subsection{Integrability of the Hill form}
We are now in conditions to study the integrability of the Hill
form.

\begin{theorem} \label{Theorem:integralhillform}
The Hill form
\begin{eqnarray*} \Omega_2 = -ydx+ (x^2 +y)dy
\end{eqnarray*}
admits a first integral  $F$ of the form
\begin{eqnarray*}
F(x,y) = \frac{2x Y_0(2\sqrt{y}) -
2\sqrt{y}Y_1(2\sqrt{y})}{x J_0(2\sqrt{y}) - \sqrt{y}J_1(2\sqrt{y})}
\end{eqnarray*}
where  $J_0,J_1$  are the  Bessel functions of first type and
$Y_0,Y_1$ are the Bessel functions of second type. The function $F$
is well-defined complex analytic  in open subsets of the form $(x,y)
\in \mathbb C \times (\mathbb C\setminus L)$ where $L$ is a closed
segment of line starting from the origin of $\mathbb C$.
\end{theorem}
\begin{proof}
First we write $\Omega_2=0$ as $ydx= (x^2 + y ) dy $ and therefore
as
\begin{equation}
\label{eq:Hillform} \frac{dy}{dx} = \frac{y}{x^2+ y}
\end{equation}
 This ODE is similar to some of the models treated in
\cite{handbook} page 375, \S  13.2.3. Indeed, the change of
coordinates $y = e^t$ leads to the model
\[
x^\prime = \frac{dx}{dt}= \frac{dx}{dy} \frac{dy}{dt} = \frac{x^2 +
y} {y} \frac{dx}{dt} = \frac{x^2 + e^t} { e^t} e^t = x^2 + e^t
\]
By its turn the ODE $x^\prime =  x^2 + e^t$ can be solved by methods
similar to those in  \cite{handbook} 13.2.2.5 and leads us to the
following:
\begin{Claim}
The solutions of \eqref{eq:Hillform} are given by $F(x,y)=c\in
\mathbb C$ where
\[
F(x,y) = \frac{2x Y_0(2\sqrt{y}) -  2\sqrt{y}Y_1(2\sqrt{y})}{x
J_0(2\sqrt{y}) - \sqrt{y}J_1(2\sqrt{y})}
\]
\end{Claim}

\begin{proof}

We shall first calculate the partial derivatives of $F$ . The first
is obtained by a standard computation
\begin{eqnarray*}
\frac{\partial F}{\partial x}(x,y) & =& \frac{2 Y_0(2\sqrt{y})(x
J_0(2\sqrt{y}) - \sqrt{y}J_1(2\sqrt{y})) -  J_0(2\sqrt{y})(2x
Y_0(2\sqrt{y}) -
2\sqrt{y}Y_1(2\sqrt{y}))}{(x J_0(2 \sqrt{y}) - \sqrt{y} J_1(2 \sqrt{y}))^2}\\
&=& \frac{2 \sqrt{y} \left( J_0(2\sqrt{y}) Y_1(2\sqrt{y}) -
J_1(2\sqrt{y}) Y_0(2 \sqrt{y})\right)}{(x J_0(2 \sqrt{y}) - \sqrt{y}
J_1(2 \sqrt{y}))^2}.
\end{eqnarray*}
\\
\\
\textbf{The partial derivative $\frac{\partial F}{\partial y}$:}
\\
\\
In this case we shall make use of the properties of the Bessel
functions mentioned above. Let us first analyze each term of the
derivative. From Proposition~\ref{PropBessel} we have
\begin{eqnarray*}
\frac{\partial}{\partial y}[x J_0(2\sqrt{y})]] &=& \frac{x}{\sqrt{y}}[J_{-1}(2\sqrt{y}) - J_{1}(2\sqrt{y})] \\
&=& \frac{x}{\sqrt{y}}[- J_{1}(2\sqrt{y}) - J_{1}(2\sqrt{y})] \\
&=& -2\frac{x}{\sqrt{y}}J_{1}(2\sqrt{y}).\\
\end{eqnarray*}

\begin{eqnarray*}
\frac{\partial}{\partial y}[\sqrt{y}J_1(2\sqrt{y})] &=& \frac{1}{2\sqrt{y}}J_1(2\sqrt{y}) + 2\sqrt{y}\frac{\partial[J_1(2\sqrt{y})]}{\partial y}\\
&=&  \frac{1}{2\sqrt{y}}J_1(2\sqrt{y}) + 2\sqrt{y}\left[\frac{1}{2\sqrt{y}}[J_0(2\sqrt{y}) - J_2(2\sqrt{y})]\right]\\
&=& \frac{1}{2\sqrt{y}}J_1(2\sqrt{y}) + J_0(2\sqrt{y}) - \left[\frac{1}{2\sqrt{y}}J_1(2\sqrt{y}) -  J_0(2\sqrt{y})\right]\\
&=& 2J_0(2\sqrt{y}).\\
\end{eqnarray*}
Similarly,
\begin{eqnarray*}
\frac{\partial}{\partial y}[2xY_0(2\sqrt{y})] =
-2\frac{x}{\sqrt{y}}Y_{1}(2\sqrt{y}),\, \, \,
\frac{\partial}{\partial y}[2\sqrt{y} Y_1(2\sqrt{y})] =
2Y_0(2\sqrt{y}).
\end{eqnarray*}

Thence
\begin{eqnarray*}
\frac{\partial}{\partial y}(x J_0(2\sqrt{y}) - \sqrt{y}J_1(2\sqrt{y})) =  - 2\frac{x}{\sqrt{y}}J_{1}(2\sqrt{y}) - 2J_0(2\sqrt{y})
\end{eqnarray*}
and,
\begin{eqnarray*}
\frac{\partial}{\partial y}(x Y_0(2\sqrt{y}) - \sqrt{y}Y_1(2\sqrt{y})) = - 2\frac{x}{\sqrt{y}}Y_{1}(2\sqrt{y}) - 2Y_0(2\sqrt{y}) .
\end{eqnarray*}
Henceforth,
\begin{eqnarray*}
&& \frac{\partial}{\partial y}\left[ \frac{2x Y_0(2\sqrt{y}) - 2\sqrt{y}Y_1(2\sqrt{y})}{x J_0(2\sqrt{y}) - \sqrt{y}J_1(2\sqrt{y})}\right] \\
\\
&=& \frac{N}{(x J_0(2\sqrt{y}) - \sqrt{y}J_1(2\sqrt{y}))^2}
\end{eqnarray*}
where by its turn
\begin{eqnarray*}
N &=& [- 2Y_0(2\sqrt{y}) - 2\frac{x}{\sqrt{y}}Y_{1}(2\sqrt{y})][x J_0(2\sqrt{y}) - \sqrt{y}J_1(2\sqrt{y})] \\
&-& [- 2J_0(2\sqrt{y}) - 2\frac{x}{\sqrt{y}}J_{1}(2\sqrt{y})][x Y_0(2\sqrt{y}) - \sqrt{y}Y_1(2\sqrt{y})] \\
\\
&=& -2xY_0(2\sqrt{y})J_0(2\sqrt{y}) + 2\sqrt{y}Y_0(2\sqrt{y})J_1(2\sqrt{y}) - 2\frac{x^2}{\sqrt{y}}Y_{1}(2\sqrt{y})J_0(2\sqrt{y})  \\
&+& 2xY_{1}(2\sqrt{y})J_1(2\sqrt{y}) + 2xJ_0(2\sqrt{y})Y_0(2\sqrt{y}) - 2\sqrt{y}J_0(2\sqrt{y})Y_1(2\sqrt{y}) \\
&+& 2\frac{x^2}{\sqrt{y}}J_{1}(2\sqrt{y})Y_0(2\sqrt{y}) - 2xJ_{1}(2\sqrt{y})Y_1(2\sqrt{y})\\
\\
&=& 2\left[\sqrt{y}+ \frac{x^2}{\sqrt{y}}\right]\left[J_1(2\sqrt{y})Y_0(2\sqrt{y}) - Y_{1}(2\sqrt{y})J_0(2\sqrt{y}) \right]
\end{eqnarray*}
\\

The exterior derivative $dF$ is then given by
\begin{eqnarray*}
dF &=& \frac{\partial F}{\partial x}dx +
\frac{\partial F}{\partial y}dy \\
&=& \left[\frac{2 \sqrt{y} \left( J_0(2\sqrt{y})
 Y_1(2\sqrt{y}) - J_1(2\sqrt{y}) Y_0(2 \sqrt{y})\right)}{(x J_0(2 \sqrt{y}) -
  \sqrt{y} J_1(2 \sqrt{y}))^2} \right]dx + \left[\frac{N}{(x J_0(2\sqrt{y}) -
  \sqrt{y}J_1(2\sqrt{y}))^2}\right]dy.
\end{eqnarray*}

Therefore,
\begin{eqnarray*}
\Omega \wedge dF &=& [-ydx + (x^2 +y)dy] \wedge \left[ \frac{\partial F}{\partial x}dx + \frac{\partial F}{\partial y}dy \right]\\
&=& -y\frac{\partial F}{\partial y}dx \wedge dy + (x^2 +y)\frac{\partial F}{\partial x}dy \wedge dx \\
&=& \left[ -y\frac{\partial F}{\partial y} - (x^2 +y)\frac{\partial F}{\partial x}\right] dx \wedge dy.
\end{eqnarray*}
Now we observe that
\begin{eqnarray*}
(x^2 +y)\frac{\partial F}{\partial x} &=& (x^2 +y) \frac{2 \sqrt{y} \left( J_0(2\sqrt{y}) Y_1(2\sqrt{y}) - J_1(2\sqrt{y}) Y_0(2 \sqrt{y})\right)}{(x J_0(2 \sqrt{y}) - \sqrt{y} J_1(2 \sqrt{y}))^2} \\
&=& -2\sqrt{y}(x^2 +y) \frac{ \left(J_1(2\sqrt{y}) Y_0(2 \sqrt{y}) - J_0(2\sqrt{y}) Y_1(2\sqrt{y})\right)}{(x J_0(2 \sqrt{y}) - \sqrt{y} J_1(2 \sqrt{y}))^2} \\
\end{eqnarray*}
and
\begin{eqnarray*}
y\frac{\partial F}{\partial y}  &=&  2y\left[\sqrt{y}+ \frac{x^2}{\sqrt{y}}\right]\left[J_1(2\sqrt{y})Y_0(2\sqrt{y}) - Y_{1}(2\sqrt{y})J_0(2\sqrt{y}) \right] \\
 &=&  2\sqrt{y}(x^2 +y) \frac{ \left(J_1(2\sqrt{y}) Y_0(2 \sqrt{y}) - J_0(2\sqrt{y}) Y_1(2\sqrt{y})\right)}{(x J_0(2 \sqrt{y}) - \sqrt{y} J_1(2 \sqrt{y}))^2}.
\end{eqnarray*}
Then we can conclude that
\begin{eqnarray*}
\Omega_2 \wedge dF &=& \left[ -y\frac{\partial F}{\partial y} -
(x^2 +y)\frac{\partial F}{\partial x}\right] dx \wedge dy \\
 &=& 0.
\end{eqnarray*}
The claim is proved.
\end{proof}
The claim proves the theorem.
\end{proof}

We apply the above result back to the Hill foliation and prove its
integrability.

\begin{Corollary} \label{TeoHill}
The Hill equation $u^{\prime \prime} + e^z u=0$ is s-integrable.
Indeed, the Hill foliation $\H: \, \omega = -ydx + xdy + [y^2 +
e^zx^2]dz =0$ admits the first integral
\begin{eqnarray} \label{IntPriHill}
H(x,y,z) = \frac{2y Y_0(2e^{\frac{z}{2}}) -
2xe^{\frac{z}{2}}Y_1(2e^{\frac{z}{2}})}{y J_0(2e^{\frac{z}{2}}) -
xe^{\frac{z}{2}}J_1(2e^{\frac{z}{2}})}.
\end{eqnarray}
\end{Corollary}
\begin{proof}
It follows from the preceding theorem and from the fact that the
Hill foliation $\H$ in $\mathbb C^3$ is obtained as the pull-back of
the foliation $\Omega_2=0$ by the map  $\Pi(x,y,z) =
(\frac{y}{x},e^z)$. This last admits the  first integral
\[
F(x,y) = \frac{2x Y_0(2\sqrt{y}) - 2\sqrt{y}Y_1(2\sqrt{y})}{x
J_0(2\sqrt{y}) - \sqrt{y}J_1(2\sqrt{y})}.
\]
So we may obtain a first integral $H=\Pi^*F$ for $\H$ which is of
the form
\begin{eqnarray*}
H=\Pi^* F(x,y) &=& F\left(\frac{y}{x},e^z\right) = \frac{2\frac{y}{x} Y_0(2e^{\frac{z}{2}}) - 2e^{\frac{z}{2}}Y_1(2e^{\frac{z}{2}})}{\frac{y}{x} J_0(2e^{\frac{z}{2}}) - e^{\frac{z}{2}}J_1(2e^{\frac{z}{2}})}\\
             &=& \frac{2y Y_0(2e^{\frac{z}{2}}) - 2xe^{\frac{z}{2}}Y_1(2e^{\frac{z}{2}})}{y J_0(2e^{\frac{z}{2}}) - xe^{\frac{z}{2}}J_1(2e^{\frac{z}{2}})}.
\end{eqnarray*}

\end{proof}

\section{The Hill foliation}

As for the moment we have the following: (1) Starting with the Hill
equation $u^{\prime \prime} + e^z u =0$ in $\mathbb C^2$ we can
consider the vector field $X(x,y,z) = y \frac{\partial}{\partial x}
-  e^z x \frac{\partial}{\partial y} + \frac{\partial}{\partial z}.$
in $\mathbb C^3$ which corresponds to the order reduction of the
Hill equation. To the vector field $X$ we associate the integrable
one-form $\omega =-ydx + xdy + [y^2 + e^z x^2]dz$ in $\mathbb C^3$
showing the strong integrability of the Hill equation $u^{\prime
\prime} + e^z u =0$. Indeed, the one-form $\frac{1}{x^2}\omega$
above is the pull-back of the Hill form $\Omega_{2}= - y dx + (x^2 +
y)dy$ by the rational map $\Pi=(-y/x,e^z)$. By its turn the Hill
form $\Omega_2$ admits a first integral of the form $F=\frac{2x
Y_0(2\sqrt{y}) - 2\sqrt{y}Y_1(2\sqrt{y})}{x J_0(2\sqrt{y}) -
\sqrt{y}J_1(2\sqrt{y})}$ given by Bessel first and second type
functions.

\subsection{Non-existence of a Liouvillian first integral}

Our aim is to show that the Hill form $\Omega_2$ admits no
Liouvillian first integral on $\mathbb CP^2$. Firstly, we
investigate the existence of invariant algebraic curves.
\begin{Lemma} \label{CurvInv}
The only invariant algebraic curves for the Hill form $\Omega_2$ on
$\mathbb CP^2$ are the line at infinity and the line  $\{y = 0\}$.
\end{Lemma}
\begin{proof}
First of all, the lines $(x=0)$ and the line at infinity are
invariant, by straightforward computation. Let us prove that there
is no other affine invariant algebraic curve. For this observe that
the leaves of $\Omega_2$ are transverse to the horizontal
$3$-dimensional cylinders $C_r : |y|=r, \, x \in \mathbb C, r
>0$. We may therefore investigate the existence of periodic orbits
for the flow $\mathcal L_r$ induced by $\Omega_2$ on $C_r$
\footnote{Alternatively we may consider the vertical cylinders $V_r:
|x|=r, |y| < \epsilon$ for $\epsilon >0$ small enough. On each such
cylinder $V_r$ we have n induced transversely complex analytic flow
$\V_r$ admitting a periodic orbit $\gamma$ obtained by the
intersection $V_r\cap (y=0)$. This periodic orbit has  a holonomy
map which is a parabolic complex diffeomorphism. By investigating
the  periodic points of such map we may conclude.}. We  evoke a
result of H. Zoladeck improving former results of other authors:

\begin{Proposition}[cf. \cite{zoladeck} Proposition~4 pp. 166-167]
 For the system  $dz/dt=z^2+re^{it}$  there exists a sequence $r_j \to \infty$
of bifurcation  values such that  for any $r\ne r_j$  this equation
has exactly one periodic solution (of period $2\pi$) and, for
$r=r_j$, the equation  does not have any bounded periodic solution.
\end{Proposition}

Our desired conclude  is then basically  a consequence of the above
proposition. Indeed, it is enough to consider the two dimensional
real ODE obtained from the original equation $- ydx + (x^2 + y)
dy=0$  by making $y= re^{it}$ which leads to the complex ODE, \,
$x^\prime =i( x^2 + re^{it})$ corresponding to the restriction to
$C_r$.
\end{proof}

We shall need the following result due to M. Singer:
\begin{Theorem}[\cite{singer} Theorem~1] \label{TeoSinger}
A polynomial differential equation $P dy - Q dx=0$ with complex
coefficients admits a Liouvillian first integral iff it the 1-form
$\omega=P dy - Q dx$ admits an integrating factor of the form
$R=\exp \int U dx + V dy$ where $U, V$ are rational functions
satisfying $\frac{\partial U}{\partial y} = \frac{\partial
V}{\partial x}$.
\end{Theorem}

If we put in the above statement $\omega= Pdy - Qdx$ and  $\eta= U
dx + V dy$ then Singer's result says that: {\em the existence of a
Liouvillian first integral for $\omega$ corresponds to the existence
of a closed rational 1-form $\eta$ which satisfies $d \omega= \eta
\wedge \omega$. In this case, the first integral is of the form
$F=\int \alpha$ where $\alpha =\Omega/\exp\int \eta$. }

Using now that the affine poles of a 1-form $\eta$ as above must be
invariant algebraic curves for $\omega=Pdy - Qdx$ (cf.
\cite{camacho-scardua} Lemma~1) for $\Omega$ we can prove:
\begin{theorem}
\label{Theorem:noliouvillian} The Hill form
\[
\Omega_2 = -ydx+ [x^2 +y]dy
\]
admits no first integral of Liouvillian type.
\end{theorem}
\begin{proof}
According to the above discussion it is enough to prove that there
is no closed  rational 1-form $\eta$ satisfying
\begin{eqnarray}\label{expLiouv}
d\Omega_2 = \eta \wedge \Omega_2.
\end{eqnarray}
Let us write  $\eta = A(x,y) dx + B(x,y) dy$. Then
\begin{eqnarray*}
d\Omega_2 &=&  -dy \wedge dx + 2x dx \wedge dy =  (1+2x) dx \wedge
dy
\end{eqnarray*}
and,
\begin{eqnarray*}
\eta\wedge\Omega_2 &=&  [Adx + Bdy] \wedge [-ydx+ [x^2 + y]dy] =
[[x^2 + y]A  + yB ]dx\wedge dy.
\end{eqnarray*}
Hence, from (\ref{expLiouv}) we get
\begin{eqnarray} \label{expsemY}
(1 + 2x) = (x^2 + y)A  + yB.
\end{eqnarray}
By Lemma~\ref{CurvInv}, the affine poles of $\eta$ are contained in
the line $\{y = 0\}$. Since $\eta$ is closed, by the Integration
lemma (\cite{Camacho-LinsNeto-Sad}) we can write
\begin{eqnarray*}
\eta &=& \lambda \frac{dy}{y} + d\left(\frac{Q}{y^n}\right).
\end{eqnarray*}
 for some $\lambda \in \mathbb{C}$ and some irreducible polynomial
  $Q(x,y)$ of degree $n \in \mathbb{N}$. Hence,
\begin{eqnarray*}
\eta &=& \lambda \frac{dy}{y} + d\left(\frac{Q}{y^n}\right)\\
     &=& \lambda \frac{dy}{y} +\frac{1}{y^n}\left(\frac{\partial Q}{\partial x}dx + \frac{\partial Q}{\partial y}dy\right) + Q\left(-\frac{n}{y^{n+1}}\right) dy\\
     &=& \frac{1}{y^n}\frac{\partial Q}{\partial x}dx + \left( \frac{\lambda}{y} + \frac{1}{y^n}\frac{\partial Q}{\partial y} - \frac{nQ}{y^{n+1}} \right)dy.
\end{eqnarray*}
It follows from (\ref{expsemY}) that
\begin{eqnarray*}
1+2x = \frac{x^2 + y}{y^n}\frac{\partial Q}{\partial x} + \lambda + y\left( \frac{1}{y^n}\frac{\partial Q}{\partial y} - \frac{nQ}{y^{n+1}}\right).
\end{eqnarray*}
Then $\lambda = 1$. By writing  $Q(x,y) = \sum\limits_{i= 0}^{n}
p_i(y)x^i$ we get
\begin{eqnarray} \label{eqPol}
2x &=& \frac{x^2 + y}{y^n}\sum_{i\geq 1} ip_i(y)x^{i-1} + \frac{1}{y^{n-1}}\sum_{i\geq 0} p'_i(y)x^i - \frac{n}{y^{n}}\sum_{i\geq 0} p_i(y)x^i.
\end{eqnarray}
\textbf{Case $n = 0$} : Equation (\ref{eqPol}) is given by
$x^2p_1(y) + yp_1(y) + y p'_0(y) = 2x.$ Analyzing the coefficients
of $x^2$ we get $p_1 \equiv 0$. Then, equation (\ref{eqPol}) would
write as $yp'_0(y) = 2x$. Therefore, this case cannot occur.
\\
\\
\textbf{Case $n = 1$} : Equation (\ref{eqPol}) is given by $
\frac{p_1(y)}{y}x^2  + p_1(y) + 2\frac{p_2(y)}{y}x^3 + 2p_2(y)x +
p'_0(y) + p'_1(y)x = 2x.$ Analyzing the coefficients of  $x^2$ and
$x^3$ we conclude that $p_1,p_2 \equiv 0$. This implies  $p'_0(y) =
2x$. Again this case cannot occur.
\\
\\
\textbf{Case $n = 2$} : Equation (\ref{eqPol}) is given by
$\frac{x^2 + y}{y^2}[p_1(y) + 2xp_2(y) + 3x^2p_3(y)] +
\frac{1}{y}[p'_0(y) + xp'_1(y) + x^2p'_2(y)] - \frac{2}{y^2}[p_0(y)
+ xp_1(y) + x^2p_2(y)] = 2x.$ Analyzing each power of  $x$ we obtain

$x^0: \frac{y}{y^2}p_1(y) + \frac{1}{y}p'_0(y) -
\frac{2}{y^{2}}p_0(y) = 0$

$x:  \frac{2y}{y^2}p_2(y) + \frac{1}{y}p'_1(y) -
\frac{2}{y^{2}}p_1(y) = 2$

$x^2: \frac{1}{y^2}p_1(y) + \frac{3y}{y^2}p_3(y) +
\frac{1}{y}p'_2(y) - \frac{2}{y^{2}}p_2(y) = 0$

$x^3: \frac{2}{y^2}p_2(y)  = 0$

$x^4: \frac{3}{y^2}p_3(y)  = 0$.

The above equations imply   $ p_1,p_2,p_3 \equiv 0$. Then, we obtain
$\frac{1}{y}p'_0(y) - \frac{2}{y^{2}}p_0(y) = 2x.$
Hence this case
does not occur.
\\
\textbf{Case $n = 3$} : Equation  (\ref{eqPol}) is given by
$\frac{x^2 + y}{y^3}[p_1(y) + 2xp_2(y) + 3x^2p_3(y) + 4x^3p_4] +
\frac{1}{y^2}[p'_0(y) + xp'_1(y) + x^2p'_2(y) + x^3p'_3(y)] -
\frac{2}{y^3}[p_0(y) + xp_1(y) + x^2p_2(y) + x^3p_3(y)] = 2x.$
Analyzing each power of $x$ we obtain

$x^0: \frac{y}{y^3}p_1(y) + \frac{1}{y^2}p'_0(y) -
\frac{3}{y^{3}}p_0(y) = 0$

$x:  \frac{2y}{y^3}p_2(y) + \frac{1}{y^2}p'_1(y) -
\frac{3}{y^{3}}p_1(y) = 2 $

$x^2:  \frac{1}{y^3}p_1(y) + \frac{3y}{y^2}p_3(y) +
\frac{1}{y^2}p'_2(y) - \frac{3}{y^{3}}p_2(y) = 0$

$x^3:    \frac{2}{y^3}p_2(y) + 4\frac{y}{y^3}p_4 + \frac{1}{y^2}p'_3
- \frac{3}{y^3}p_3  = 0$

$x^4:  \frac{3}{y^3}p_3(y)  = 0$

$x^5:  \frac{4}{y^3}p_4(y)  = 0$.

The above equations imply that  $ p_1,p_2,p_3,p_4 \equiv 0$. Then,
we obtain $\frac{1}{y^2}p'_0(y) - \frac{3}{y^{3}}p_0(y) = 2x.$
Therefore this case is not possible.
\\
\\
\textbf{Case $n > 2$} : Equation (\ref{eqPol}) writes
\begin{eqnarray*}
&&\frac{x^2 + y}{y^n}[p_1(y) + 2xp_2(y)+ \dots + (n-1)x^{n-2}p_{n-1}(y) + nx^{n-1}p_n(y) + (n+1)x^{n}p_{n+1}(y)] \\
&+& \frac{1}{y^{n-1}}[p'_0(y) + xp'_1(y) + x^2p'_2(y) + \dots + x^{n-2}p'_{n-2}(y) + x^{n-1}p'_{n-1}(y) +  x^np'_n(y)] \\
&-&\frac{n}{y^n}[p_0(y) + xp_1(y) + x^2p_2(y) + \dots + x^{n-2}p_{n-2}(y)+ x^{n-1}p_{n-1}(y) +  x^np_n(y)] = 2x.
\end{eqnarray*}
Again, the analysis of each power of $x$ gives

$x^0: \frac{y}{y^n}p_1(y) + \frac{1}{y^{n-1}}p'_0(y) -
\frac{n}{y^{n}}p_0(y) = 0$

$x:  \frac{2y}{y^{n-1}}p_2(y) + \frac{1}{y^{n-1}}p'_1(y) -
\frac{n}{y^{n}}p_1(y) = 2$

$x^2: \frac{1}{y^n}p_1(y) + \frac{3y}{y^n}p_3(y) +
\frac{1}{y^{n-1}}p'_2(y) - \frac{n}{y^{n}}p_2(y) = 0$

$x^{n-1}:  \frac{(n-2)}{y^n}p_{n-2}(y) + \frac{ny}{y^n}p_{n}(y) +
\frac{1}{y^{n-1}}p'_{n-1}(y) - \frac{n}{y^n}p_{n-1}(y) = 0$

$x^n: \frac{(n-1)}{y^n}p_{n-1}(y) + \frac{(n+1)y}{y^n}p_{n+1}(y) +
\frac{1}{y^{n-1}}p'_n(y) - \frac{n}{y^n}p_n(y) = 0$

$x^{n+1}: \frac{n}{y^n}p_{n}(y) = 0$

$x^{n+2}: \frac{n+1}{y^n}p_{n+1}(y) = 0.$

The above equations imply  $p_1,\dots, p_{n-1},p_n,p_{n+1} \equiv
0$. Then, equation (\ref{eqPol}) is given by
\[
\frac{1}{y^{n-1}}p'_0(y) - \frac{n}{y^{n}}p_0(y) = 2x.
\]
Thus, the general case  $n>2$ is also excluded. This ends the proof.
\end{proof}

\subsection{Reduction of singularities}

In what follows we make use of the blow-up technique in order to
extract more information about the Hill form. For more details in
this technique we refer to  \cite{Camacho-sad inv, IlyYak 2006,
seidenberg}. Let us check which models of singularities arise after
a number of blow-ups. We consider the quadratic blow-up  at the
origin $0\in \mathbb C^2$ as pattern in what follows.
\begin{Proposition} \label{DesingOmega}
Let $$\Omega_2 =  -y dx + (x^2 + y)dy$$ be the Hill form. Then,\\
\item[(i)] The first blow-up originates two singularities
in the exceptional divisor. One singularity is of Siegel type  of
the form
\begin{eqnarray} \label{singSiegel1}
-ydx + (1 - x + x^2y)dy=0.
\end{eqnarray}
The second singularity is nilpotent and given by:
\begin{eqnarray} \label{singNilpot1}
(-y +yx + y^2)dx + (x^2 + yx)dy=0.
\end{eqnarray}
\item[(ii)] After performing  $n>1$ blow-ups starting from the
second singularity and always repeating the process on the
corresponding nilpotent singularity  we have: an exceptional divisor
$\mathbb{E} = \cup_{j=1}^{n}\mathbb{E}_j$ with a nilpotent
singularity at the intersection of the strict transform of the
$x$-axis and  $\mathbb{E}_n$, which is given by:
\begin{eqnarray} \label{singNilpotN}
(-y + nyx + nx^{n-1}y^2)dx + (x^2 + yx^n)dy.
\end{eqnarray}
Moreover, at each corner $\mathbb{E}_{j}\cap \mathbb{E}_{j+1}$, we
have a Siegel type singularity given by:
\begin{eqnarray} \label{singSiegelj}
(-y + jxy^2 + jx^{j-1}y^{j+1})dx + (-x + (j+1)x^2y +
(j+1)x^{j}y^{j})dy.
\end{eqnarray}
\end{Proposition}

The proof of Proposition~\ref{DesingOmega} is a straightforward
computation with blow-ups and we chose to omit it. For mode details
in the subject of reduction of singularities we refer to
\cite{Camacho-sad inv} or \cite{seidenberg}.

Next we give the explicit first integrals for the models arising in
the reduction of singularities of the Hill form as described in
Proposition~\ref{DesingOmega}.

\begin{Proposition} \item[(i)] Let $n \geq 1$ be given. The
singularity
$$
(-y + nyx + nx^{n-1}y^2)dx + (x^2 + yx^n)dy=0
$$
admits a first integral
\begin{eqnarray*}
F_n(x,y) = \frac{2x Y_0(2x^{\frac{n}{2}}y^{\frac{1}{2}}) -
2x^{\frac{n}{2}}y^{\frac{1}{2}}Y_1(2x^{\frac{n}{2}}y^{\frac{1}{2})}}{x
J_0(2x^{\frac{n}{2}}y^{\frac{1}{2}}) -
x^{\frac{n}{2}}y^{\frac{1}{2}}J_1(2x^{\frac{n}{2}}y^{\frac{1}{2}})}
\end{eqnarray*}

\item[(ii)] Let $j \in \{1,\dots,n-1\}$ be given. The Siegel type singularity
\begin{eqnarray*}
(-y + jxy^2 + jx^{j-1}y^{j+1})dx + (-x + (j+1)x^2y +
(j+1)x^{j}y^{j})dy
\end{eqnarray*}
admits the first integral
\begin{eqnarray*}
G_j(x,y) = \frac{2xy Y_0(2x^{\frac{j}{2}}y^{\frac{j+1}{2}}) -
2x^{\frac{j}{2}}y^{\frac{j+1}{2}}Y_1(2x^{\frac{j}{2}}y^{\frac{j+1}{2}})}{xy
J_0(2x^{\frac{j}{2}}y^{\frac{j+1}{2}})-
x^{\frac{j}{2}}y^{\frac{j+1}{2}}J_1(2x^{\frac{j}{2}}y^{\frac{j+1}{2}})}.
\end{eqnarray*}

\end{Proposition}
\begin{proof} Let
$\pi_n(x,y) = (x,x^ny), \, \Omega_2 = -ydx + (y + x^2)dy$ then
$\pi_n^*\Omega = x^{n}[(-y + nyx + nx^{n-1}y^2)dx + (x^2 +
yx^n)dy].$ Now we apply  Theorem~\ref{Theorem:integralhillform} and
obtain (i).
 Now, let us define  $\pi(x,y) = (xy,y)$. Let
$$\xi_{j} = (-y + jyx + jx^{j-1}y^2)dx + (x^2 + yx^{j})dy.$$
Then $\pi^*\xi_{j}  = y[(-y + jxy^2 + jx^{j-1}y^{j+1})dx + (-x +
(j+1)x^2y + (j+1)x^{j}y^{j})dy)dy]$.  From  this pull-back together
with  (i) we get (ii).
\end{proof}

Let us introduce a class of functions that may be useful. Denote by
$A(\mathcal J)$ the algebra of functions generated by compositions
of first and second type (one variable) Bessel functions and complex
valued algebraic functions of two complex variables. Also denote by
$K(\mathcal J)$ the field of fractions of $A(\mathcal J)$. Finally,
we have the following definition:

\begin{Definition}
{\rm  We shall say that a function $F(x,y)$ of two complex variables
is {\it generated by Bessel functions} if it belongs to the field
$K(\mathcal J)$. }
\end{Definition}

\begin{Question}
Let $A(x,y)dx + B(x,y)dy=0$ be a germ of Siegel type singularity at
the origin $0 \in \mathbb C^2$. Under what conditions can we assure
the existence of a first integral  generated by Bessel functions?
\end{Question}

Notice that a function $F(x,y)\in K(\mathcal J)$ always defines a
uniform complex analytic function in some dense open subset of the
complex plane $\mathbb C^2$.

\subsection{Holonomy} If we consider the Hill form $\Omega_2=-ydx +
(x^2 + y)dy$ then we have a singularity at the origin $0\in \mathbb
C^2$. This germ of foliation admits a separatrix given by $\Gamma:
(y=0)$. We may ask about the holonomy map of this separatrix. If we
consider a simple loop $\gamma: x(t)=r_0 e^{it}, \, 0 \leq t \leq
2\pi$ in $\Gamma$ then we may consider the holonomy map of $\Gamma$
with respect to the transverse section $\Sigma: (x=r_0)$ as follows:
let $q_0$ be the point $(r_0,0)$. The holonomy map $h\colon (\Sigma,
q_0) \to (\Sigma,q_0)$ is given by $h(y_0)=y(2 \pi, y_0)$ where
$y(t,y_0)$ is the solution of the ODE
\[
\frac{dy}{dt}=\frac{y}{ r_0^2 e^{2 it} + y} . r i e^{it}
\]
that satisfies $y(0,y_0)=y_0$. The first integral $F(x,y)= \frac{2x
Y_0(2\sqrt{y}) - 2\sqrt{y}Y_1(2\sqrt{y})}{x J_0(2\sqrt{y}) -
\sqrt{y}J_1(2\sqrt{y})}$ may be used to calculate (study) the
holonomy map $h$, for it satisfies the relation
$F(r_0,h(y))=F(r_0,y)$. This is an expression involving Bessel type
functions.

All this suggests that there may be a theory of foliations admitting
Bessel type functions as first integrals, in terms of their holonomy
groups.

\subsection{The general case}

Now we turn our attention to the general case of the Hill equation:
$u^{\prime \prime } + p(z)u=0$ where $p(z)$ is a periodic complex
analytic function defined in the complex plane. We may assume that
$p$ is not constant and of period $2 \pi i$. By ordinary covering
spaces theory we can write $p(z)=f(e^z)$ for some complex analytic
function $f \colon \mathbb C\setminus \{0\} \to \mathbb C$. The
corresponding Hill foliation is then given by $\omega=-ydx + xdy
+[y^2 + p(z)x^2]dz=-ydx + xdy +[y^2 + f(e^z)x^2]dz$. We may rewrite
\[
\omega/x^2 = (- y dx + x dy)/x^2 + [(y/x)^2 + f(e^z)]dz
\]
If we put $\vr=-y/x$ and $\psi=e^z$ then we have $dz=d\psi/\psi$ and
then
\[
\omega/x^2= - d \phi + ( \phi^2 + f(\psi))d\psi/\psi
\]
This shows that the foliation is the pull-back by the map
$(\phi,\psi)$ of the two dimensional model $\theta_2=-ydx + (x^2 +
f(y))dy=0$.


\section{Integrability and classical solutions}

G. W. Hill has constructed periodic solutions for
(\ref{HEqComplex}) in the form of trigonometric series
\begin{eqnarray*}
u(z) = \sum_{k = 0}^{\infty} A_{2k + 1} \cos ((2k + 1)z), \,  v(z) =
\sum_{k = 0}^{\infty} B_{2k + 1} \sin ((2k + 1)z).
\end{eqnarray*}
where the coefficients  $A_{2k + 1},B_{2k + 1}$ are represented by
power series with respect to $m$. Hill has not studied their
convergence though. The first attempt to estimate the convergence
radius was done by  Lyapunov
 \cite{lyapuConv}.

Let us explore some of these ideas in our framework. We start by the
following useful result.

\begin{Theorem} \label{Teo: relation}
Let $w_1,w_2$ be two linearly independent solutions of the Hill
equation $u'' + p(z)u = 0$ defined in a strip $A \subseteq \mathbb
C$. Then
\begin{eqnarray*}
H(x,y,z) = \frac{xw_1'(z) - yw_1(z)}{xw_2'(z) - yw_2(z)}
\end{eqnarray*}
is a first integral for the Hill foliation
\begin{eqnarray*}
\omega_\H = -ydx + xdy + [y^2 + p(z)x^2]dz = 0
\end{eqnarray*}
defined and meromorphic in $A$ without indefinite points.
\end{Theorem}
\begin{proof}
First observe that $H$ is not constant. Indeed, $H=c \in \mathbb C$
implies $x(w_1 ^\prime(z) - c w_2 ^\prime(z))=y(w_1(z) - c w_2(z))$
and then $w_1(z)=cw_2(z)$ which is a contradiction since $w_1$ and
$w_2$ are linearly independent. Now we assume that there is a (real
analytic nondegenerate) curve $\Gamma$ of indefinite points for $H$.
Then we have $\frac{w_1
^\prime}{w_1}=\frac{y}{x}=\frac{w_2^\prime}{w_2}$ in this curve.
This implies that $\ln w_1 = \ln w_2 + c$ for some constant and
therefore $w_1= k w_2$ for some constant $k$ along this curve. Since
$w_1$ and $w_2$ are one variable complex analytic functions this
implies that $w_1=kw_2$ identically, yielding another contradiction.
This shows that the set of indefinite points of $H$ is discrete.
Since we are in dimension $3$ this shows that $H$ is indeed free of
indefinite points (Hartogs' theorem for instance). Let us now prove
that $H$ is actually a first integral for $\mathcal H: -ydx + xdy +
[y^2 + p(z)x^2]dz = 0$. We start with the model in dimension two,
ie., the Hill form
\begin{eqnarray*}
\Omega_\H = -dx + [x^2 + p(y)]dy.
\end{eqnarray*}
We put
\begin{eqnarray*}
F(x,y) = \frac{w_1'(y) + xw_1(y)}{w_2'(y) + xw_2(y)}.
\end{eqnarray*}
Now we compute the partial derivatives of $F$ with respect to $x$
and $y$:

\begin{eqnarray*}
\frac{\partial F(x,y)}{\partial x} &=& \frac{w_1\left(w_2' +
xw_2\right)} {\left(w_2' + xw_2\right)^2} - \frac{ w_2\left(w_1' +
xw_1\right)}{(w_2' + xw_2)^2}  = \frac{w_1w_2' +xw_1w_2}
{\left(w_2') + xw_2\right)^2} - \frac{ w_2w_1' + x w_2 w_1}{(w_2' +
x w_2)^2}.
\end{eqnarray*}
Thus,
\begin{eqnarray*}
\frac{\partial F(x,z)}{\partial x} = \frac{w_1w_2' - w_2w_1'}
{\left(w_2' + xw_2\right)^2} &=& \frac{ (w_1'' + x w_1')(w_2' + x w_2)}{(w_2' + x w_2)^2} - \frac{ (w_2'' + x w_2')(w_1' + x w_1)}{(w_2' + x w_2)^2} \\
\\
&=& \frac{ (-pw_1 + x w_1')(w_2' + x w_2)}{(w_2' + x w_2)^2} - \frac{ (-p w_2 + x w_2')(w_1' + x w_1)}{(w_2' + x w_2)^2} \\
\\
&=& (p + x^2)\frac{(w_1'w_2 - w_2'w_1)}{(w_2' + x w_2)^2}.
\end{eqnarray*}
Therefore, $dF\wedge \Omega_\H = \left(\frac{\partial F}{\partial
x}dx + \frac{\partial F}{\partial y}dy \right)\wedge \left( - dx +
(p  + x^2)dy \right) =  \left[ (p + x^2)\frac{\partial F}{\partial
x} + \frac{\partial F}{\partial y} \right] dx \wedge dy=0$.  This
shows that  $F$ is a first integral for  $\Omega_\H$. Now we use the
pull-back, ie., define  $H = \Pi^*F$, where $\Pi(x,y,z) =
(-\frac{y}{x},z)$, obtaining then $dH\wedge \omega_\H = d(\Pi^*F)
\wedge \Pi^*\Omega_\H
    = \Pi^*(dF \wedge \Omega_\H)     = 0.$
Explicitly we have
\[
H(x,y,z) = \frac{w_1'(z) -\frac{y}{x}w_1(z)}{w_2'(z)
-\frac{y}{x}w_2(z)} = \frac{xw_1'(z) -yw_1(z)}{xw_2'(z) - yw_2(z)}
\]
is a first integral for  $\omega_\H$.
\end{proof}

As a straightforward interesting consequence of the above theorem is
the  {\it Bessel-Hill} case below:
\begin{Corollary}[Bessel-Hill form]
The {\it Bessel-Hill} foliation
\begin{eqnarray*}
\omega_\H = -ydx + xdy + [y^2 - (r^2 - k^2e^2z)x^2]dz
\end{eqnarray*}
admits the following first integral
\begin{eqnarray*}
H(x,y,z) = \frac{xY'_{r}(ke^z) - yY_{r}(ke^z)}{xJ'_{r}(ke^z) -
yJ_{r}(ke^z)}.
\end{eqnarray*}
\end{Corollary}
\begin{proof}
Indeed, this follows from the above theorem by observing that
$J_{r}(ke^z), Y_{r}(ke^z)$ are (well-known classical) solutions of
the Bessel equation  $u'' + (r^2 - k^2e^{2z})u = 0$.
\end{proof}

\begin{Remark}\rm{We shall point-out that the introduction of the
perturbation $r_2 - k^2 e^{2 z}$ instead of the function $e^{2 z}$
in the Hill equation, generates an important change in the first
integral. Indeed, the constant is directly connected to the order of
the Bessel functions in the first integral.  For instance, when  $r
= 1/2$, we have  $J_{1/2}(z) = \left( \frac{2}{\pi z}\right)^{1/2}
\sin z$ which is a Liouvillian function and the first integral is
Liouvillian as well. }\end{Remark}

\section{Laurent-Fourier type formal first integral for  Hill foliations}
We turn again our attention to the basic complex Hill equation
$u''(z) + p(z)u(z) = 0$ where  $p$ is a complex analytic periodic
function defined in a strip $A\subset \mathbb{C}$, containing the
real axis $\Im(z)=0$. We shall apply Floquet theory in order to
obtain some formal solutions. Order reduction of the problem is
given by $x=u, y=u^\prime$ and then we obtain the linear first order
problem $X^\prime = A X$ where $X=\begin{pmatrix} x \\ y
\end{pmatrix}$ and $A=\begin{pmatrix} 0 & 1
\\ -p(z) & 0\end{pmatrix}$. The characteristic equation of the
matrix $A$ is $\lambda^2 + p(z)=0$ which has symmetric roots.
Applying then  the classical  Floquet theory we conclude that
 a  solution of the Hill equation is of the form
\begin{eqnarray} \label{fi1}
\varphi_1(z) = c_1 e^{\mu z} P(z)
\end{eqnarray}
where $c_1,\mu$ are constants ($\mu$ is a characteristic exponent)
and $P(z)$ is a periodic complex valued analytic function (see  for
instance \cite{hilleBook} page 244). Moreover, another solution to
(\ref{HEqComplex}), linearly independent with respect to $\varphi_1$
is given by
\begin{eqnarray}\label{fi2}
\varphi_2(z) = c_1 e^{-\mu z} P(-z)
\end{eqnarray}
An immediate corollary of Theorem~\ref{Teo: relation} is that we
have  a first integral of the Hill foliation $\H: \omega_{\H} = -ydx
+ xdy + [y^2 + p(z)x^2]=0$, of parameter $p$ as follows
\begin{eqnarray*}
H_p(x,y,z) = \frac{x\varphi_1'(z) - y\varphi_1(z)}{x\varphi_2'(z) -
y\varphi_2(z)}.
\end{eqnarray*}
Replacing then  (\ref{fi1}) e (\ref{fi2}) in the above expression we
obtain:

\begin{Corollary}
The Hill foliation of parameter $p(z)$, periodic and complex
analytic in the strip $A\subseteq \mathbb C$, admits a first
integral of the form:
\begin{eqnarray*}
H_p(x,y,z) = e^{2\mu z}\frac{x \mu P(z) + xP'(z) - yP(z)}{-x \mu
P(-z) + xP'(-z) - yP(-z)},
\end{eqnarray*}
where $P(z)$ is a periodic complex valued analytic function and
$\mu$ is a constant.
\end{Corollary}
It is well-known that a periodic   complex analytic function $P(z)$
defined in a strip $A\subseteq \mathbb C$ containing $\Im(z)=0$
admits a Fourier series expansion. Indeed, if the period is of $P$
is $2\pi i$, we have $P(z) = \sum_{k = -\infty}^{\infty}a_ke^{kz}.$
Replacing the series of $P'(z), \, P(-z)$  and $P'(-z)$  in the
 expression of $H_p(x,y,z)$ above we obtain ($\Sigma_k$ stands for
$\sum\limits_{k=-\infty}^\infty$) :

\begin{theorem}
\label{Theorem:LaurentFourier} Let $p(z)$ be a complex analytic
function periodic in a strip $A\subseteq \mathbb C$ containing the
real axis $\Im(z)=0$. Then the corresponding Hill foliation $-ydx +
xdy + [y^2 + p(z)x^2]=0$ admits a Laurent-Fourier type formal first
integral given by an expression
\begin{eqnarray} \label{IntFormalHill} H_p(x,y,z)&=&  e^{2\mu
z}\frac{x \mu \sum_k a_ke^{ kz} + x\sum_k  k a_k e^{ kz} - y\sum_k
a_ke^{kz}}{-x \mu \sum_k a_k e^{-kz} - x\sum_k k a_k e^{-kz} -
y\sum_k a_k e^{-kz}}.
\end{eqnarray}

\end{theorem}

\begin{Remark}
{\rm The term {\it formal} in the above statement has to be
clarified. Indeed, we are talking about quotients of formal Laurent
series (ie., quotients of series with an infinite number of terms
with positive and negative exponents). The convergence of such
series must be a subject of a deeper discussion. Anyway, once we
have made the above convention, we shall refer to the expression
$H_p$ as {\it formal first integral for the Hill foliation}. The
``change of coordinates"  $ \Pi(x,y,z)=(u,v)$ where  $u =
-\frac{y}{x}$ and $v = e^z$ gives the following Laurent-type formal
first integral:
\begin{eqnarray} \label{IntFormalHill2}
H=\Pi^*F(u,v) = v^{2\mu}\frac{ \mu \sum_k a_kv^{k} + \sum_k  k a_k
v^{ k} + u\sum_k a_k v^{k}}{- \mu \sum_k a_k v^{-k} - \sum_k k a_k
v^{-k} + u\sum_k a_k v^{-k}}.
\end{eqnarray}
}
\end{Remark}


\bibliographystyle{amsalpha}

\vglue.2in

\begin{tabular}{l l}
Fernando Reis & Bruno Sc\'ardua\\
Instituto de Matem\'atica & Instituto de Matem\'atica\\
Universidade Federal do Rio de Janeiro & Universidade Federal do Rio de Janeiro\\
Caixa Postal 68530 & Caixa Postal 68530\\
CEP. 21945-970 Rio de Janeiro - RJ & CEP. 21945-970 Rio de Janeiro - RJ\\
BRASIL & BRASIL\\
fernandoppreis@gmail.com  & bruno.scardua@gmail.com
\end{tabular}

\end{document}